\documentclass[12pt,reqno]{amsart}
\usepackage[left=3cm,top=2cm,right=3cm,bottom=2cm]{geometry}
\usepackage{amssymb}
\usepackage{amsbsy}
\usepackage{epsfig}
\usepackage{tikz}

\begin{document}
\title{Counting spanning trees in double nested graphs}
\author{Fernando Tura}
\address{ Departamento de Matem\'atica, UFSM,97105--900 Santa Maria, RS, Brazil}
\email{\tt ftura@smail.ufsm.br}

\pdfpagewidth 8.5 in \pdfpageheight 11 in

\newcommand{\diagonal}[8]{
\begin{array}{| c | r |}
b_i & d_i \\
\hline
 #1 & #5 \\
        & \\
 #2 & #6 \\
        & \\
 #3  & #7 \\
         &  \\
 #4  & #8 \\
 \hline
 \end{array}
}

\newcommand{\bfx}{{\mathbf x}}
\newcommand{\casei}{{\bf case~1}}
\newcommand{\subia}{{\bf subcase~1a}}
\newcommand{\subib}{{\bf subcase~1b}}
\newcommand{\subic}{{\bf subcase~1c}}
\newcommand{\caseii}{{\bf case~2}}
\newcommand{\subiia}{{\bf subcase~2a}}
\newcommand{\subiib}{{\bf subcase~2b}}
\newcommand{\caseiii}{{\bf case~3}}
\newcommand{\myvar}{x}
\newcommand{\exvar}{\frac{\sqrt{3} + 1}{2}}
\newcommand{\Prf}{{\noindent \bf Proof: }}
\newcommand{\PrfSketch}{{\bf Proof (Sketch): }}
\newcommand{\boldQ}{\mbox{\bf Q}}
\newcommand{\boldR}{\mbox{\bf R}}
\newcommand{\boldZ}{\mbox{\bf Z}}
\newcommand{\boldc}{\mbox{\bf c}}
\newcommand{\sign}{\mbox{sign}}
\newcommand{\alphaseq}{{\pmb \alpha}_{G,\myvar}}
\newcommand{\alphaseqGprime}{{\pmb \alpha}_{G^\prime,\myvar}}
\newcommand{\alphaseqlam}{{\pmb \alpha}_{G,-\lambdamin}}
\newtheorem{Thr}{Theorem}
\newtheorem{Pro}{Proposition}
\newtheorem{Que}{Question}
\newtheorem{Con}{Conjecture}
\newtheorem{Cor}{Corollary}
\newtheorem{Lem}{Lemma}
\newtheorem{Fac}{Fact}
\newtheorem{Ex}{Example}
\newtheorem{Def}{Definition}
\newtheorem{Prop}{Proposition}
\def\floor#1{\left\lfloor{#1}\right\rfloor}

\newenvironment{my_enumerate}{
\begin{enumerate}
  \setlength{\baselineskip}{14pt}
  \setlength{\parskip}{0pt}
  \setlength{\parsep}{0pt}}{\end{enumerate}
}
\newenvironment{my_description}{
\begin{description}
  \setlength{\baselineskip}{14pt}
  \setlength{\parskip}{0pt}
  \setlength{\parsep}{0pt}}{\end{description}
}

\begin{abstract}
In this paper we  give a linear time algorithm  for determining the number of spanning trees of a double nested graph. This class of graphs is a bipartite graph (two color classes) which admits a partition on both color classes into cells with a nesting property imposed.  The algorithm proposed here takes advantage of the structure of these graphs. In this way, it works taking the values of vertices degree with an additional increment, such that the number  of spanning trees of $G$  is computed as the product of a set  of values which are associated with the vertices of $G.$   Our proofs are based on  Kirchhoff matrix tree theorem which expresses the number of spanning trees of a graph in terms of the cofactor of its Kirchhoff matrix. We finish the paper by applying the algorithm for a special cases of bipartite graphs. 
\end{abstract}


\maketitle
\section{Introduction}
\label{intro}

A spanning tree of   a connected  undirected graph $G$ on $n$ vertices is  a connected  $(n-1)$- edge subgraph of $G.$
The problem of computing the number of spanning trees on the graph $G$  is an important  problem in graph theory.
 In this context, there are a lot of papers that derive formulas and  algorithms (see  \cite{yan, zang, nikolo, hamer}). In particular, if $K_{m,n}$ is the complete bipartite graph,
it is very well known that the number of spanning tree is equal to $m^{n-1} n^{m-1}$  \cite{Abu}.

This paper is concerned  with  double nested graphs, also called bipartite chain graphs. This class of graphs plays an important role in the study of extremal graphs, a branch of graph theory, well known as spectral graph theory \cite{Bro12}.
In fact, among all  connected  bipartite graphs with fixed order and size, the graphs with maximal index (largest eigenvalue of adjacency matrix) are double nested graphs
\cite{Friedland}.

A double nested graph is a bipartite graph (two color classes) which admits a partition on both color classes into cells with a nesting property imposed. Another way to characterize this class of graphs is through  forbidden induced subgraphs. It is known that  a double nested graph  is characterized as being 
$\{ 2K_2, C_3, C_5 \}$- free graphs. Linear time algorithms for recognizing this class of graphs  are given in \cite{Heggernes}.

The goal  of this paper is to give a linear time algorithm  for determining the number of spanning trees of double nested graphs.  The algorithm proposed here is based on the linear time algorithms of diagonalization matrices associates to graphs. In \cite{JTT2013} was presented a linear time algorithm for computing a diagonal matrix congruent to $A+xI,$ where $A$ is the adjacency matrix of a threshold graph and $x \in \mathbb{R}.$  Although the main application of the diagonalization algorithm for threshold graphs is to localize the eigenvalues, it also has been used as a theoretical tool. For example, in \cite{JTT2015} the diagonalization algorithm was used to prove that there is no threshold graph
with eigenvalues in the interval $(-1,0).$  Recently,  in \cite{chain} a similar procedure was given for localizing the adjacency  eigenvalues of chain graphs.

The proofs of ours results are based on  Kirchhoff Matrix Tree Theorem \cite{ Harary}, which expresses the number of spanning trees of a graph in terms of the cofactor of its Kirchhoff matrix. 
 
 In this paper we also apply  the algorithm for some double nested graphs.
Then, for a special cases, we determine the number of spanning trees of some bipartite graphs that contain few cycles.

Here is an outline of the remainder of this paper.  In Section 2, we present some
definitions and background results.
In Section 3, we present the algorithm for determining the number of spanning trees of a double nested graph, as well as its correctness.
 In Section 4, we finish the paper by applying the algorithm for a special cases of bipartite graphs.

\section{Preliminares} \label{pre}

We consider finite undirected graphs with  non loops or multiple edges. For a graph $G,$ we denote by $V(G)$ and $E(G)$ the vertex set and edge set of $G,$ respectively. The {\em neighborhood}  $N(v)$ of a vertex $v$ of $G$ is the set of all the vertices of $G$ which are adjacent to $v.$   We use $d(v)$ to denote the degree  of vertex $v,$ that is the number of edges incident on $v,$ thus, $d(v) = N(v).$

\subsection{ Double nested graphs}

Recall  first that a bipartite graph $G$  with bipartition  $(U, V) $ is  a {\em double nested graph}  if there exist  partitions $U= U_1 \cup U_2 \cup  \ldots \cup U_h$ 
and $V = V_1\cup V_2\cup \ldots \cup V_h,$ such that $U_i$ and $V_i$ are non-empty sets, and  the neighborhood of each vertex in $U_i$ is $V_1\cup V_2\cup \ldots \cup V_{h+1-i}$ for $1\leq i \leq h.$
If $  |U_i | = m_i$  $(i =1, 2, \ldots,h)$ and $| V_i | = n_i $ $(i =1, 2, \ldots,h),$ then $G$ is denoted by $G(m_1, \ldots, m_h ; n_1, \ldots, n_h).$ Figure \ref{double 1}  shows the structure of a double nested  graph.

\begin{center}
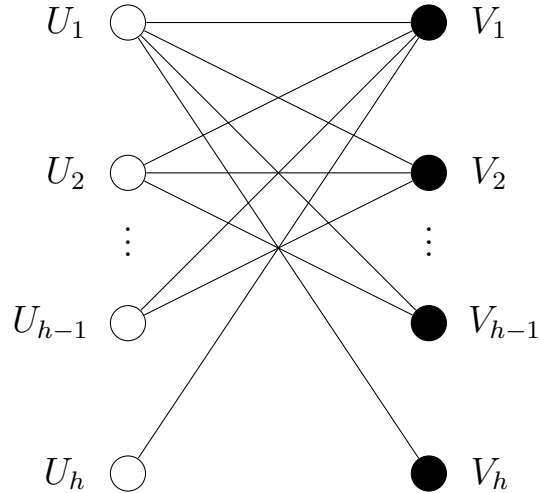
\begin{figure}[h!]     
\begin{tikzpicture}
  [scale=1,auto=left,every node/.style={circle,scale=1.2}]
  \node[draw,circle ,label=left:$U_1$] (p) at (0,6) {};
  \node[draw,circle,fill=black,label=right:$V_2$, label=below:$\vdots$] (o) at (4,4) {};
  \node[draw,circle,fill=black,label=right:$V_1$] (n) at (4,6) {};
   \node[draw,circle,label=left:$U_h$] (q) at (0,0) {};
  \node[draw,circle,label=left:$U_{h-1}$] (l) at (0,2) {};
  \node[draw,circle,fill=black,label=right:$V_h$] (k) at (4,0) {};

  \node[draw,circle, fill =black, label=right:$V_{h-1}$] (c) at (4,2) {};
  \node[draw,circle,label=left:$U_2$, label= below:$\vdots$] (d) at (0,4) {};
  \path

        (n) edge node[below]{}(p)
        (p) edge node[below]{}(o)
        (p) edge node[below]{}(c)
        (p) edge node[below]{}(k)

         (o) edge node[below]{}(l)
        (o) edge node[below]{}(d)
                
        (n) edge node[below]{}(d)
        (n) edge node[below]{}(l)
      (n) edge node[below]{}(q)

       (d) edge node[below]{}(c);        
\end{tikzpicture}
       \caption{The structure of a double nested graph.}
       \label{double 1}
\end{figure}
\end{center}

\subsection{ Kirchhoff Matrix}

Let  $G= (V,E)$ be an undirected graph with vertex set $V$ and edge set $E.$ If $|V| = n,$ the   adjacency matrix  $A(G)= (a_{ij}),$  is 
 the $n \times n$   matrix of zeros and ones  such that $a_{ij} = 1$ if 
there is an edge between $v_i$ and $v_j$, and 0 otherwise. Let $D = D(G)=diag( d(v_1), d(v_2), \ldots, d(v_n))$ be the diagonal matrix of vertex  degrees. Then the Kirchhoff
(laplacian) matrix of $G$ is 
$K = K(G) = D(G) - A(G).$ 

For an $n\times n$ matrix $A,$ the $ij$th minor is the determinant of the $(n-1)\times (n-1)$ matrix $M_{ij}$ obtained from $A$ deleting row $i$ and column $j.$ The $i$th cofactor denoted $A_i$ equals $det(M_{ii}).$ The Kirchhoff Matrix  Tree Theorem is one of most important results in graph theory. It provides a formula for the number of spanning trees of a graph $G,$ in terms of the cofactors of its  Kirchhoff Matrix. 

\begin{Thr}[Kirchhoff Matrix Tree Theorem \cite{Harary}]
\label{main}
 For any  graph $G$ with its Kirchhof matrix $K(G),$ the cofactors of $K(G)$ have the same value, and this value gives the number of spanning trees of $G.$ 
\end{Thr}

\section{ The Algorithm}
In this section  we present a linear time algorithm for determining the number of spanning trees of a double nested graph $G.$
We denote by $\tau (G)$ the number of spanning trees of $G.$

Recall that matrices are congruent if one can obtain the other by a sequence of pairs of elementary operations, each pair consisting of a row operation followed by the same  column operation.

Let $K(G)$ be the Kirchhoff matrix of a double nested graph $G(m_1, \ldots, m_h ; n_1, \ldots, n_h),$  
 $N= \sum_{k=1}^{h} n_k,$  $M= \sum_{k=1}^h m_k$ and $n=N+M.$
We assume that   $K(G)$ is $n\times n$ matrix  represented in the form

\begin{center}
$K(G)= \left[\begin{array}{ccc}
 d_{m_1}&         &       H        \\
  &  \ddots      &           \\
             H^T      &             &  d_{n_h}  \\
\end{array}\right],$
\end{center}
where $H = (H_{ij})$ for $ 1 \leq i,j \leq h$ is the block matrix defined by
\begin{equation}
\label{H}
H_{ij} =
\begin{cases}
        - J_{m_i \times n_j}& \mbox{ if }  i+j \leq h+1  \\
          O_{m_i \times n_j} & \mbox{ otherwise } .\\
\end{cases}
\end{equation}
$J_{m_i \times n_j}$ is the $m_i \times n_j$ all 1- matrix and $O_{m_i \times n_j}$ is the $m_i \times n_j$ 0-matrix.

According the  Kirchhoff Matrix Tree Theorem, we need to compute a cofactor of $K(G).$  Let  $K(G')$ be the matrix obtained from $K(G)$ deleting the last row and column. 
First, we show that $K(G')$ may be reduced to a certain tridiagonal matrix congruent, by
row and column operations. After this, we use the $LU$ decomposition for computing a cofactor of $K(G').$

The process for computing a cofactor of $K(G')$ will be divided into four steps.\\

\noindent{\bf Step 1} If $m_1=1$ then the first block is done. If $m_1 > 1$ then the first $m_1$ rows (and columns) in $H$ (and $H^T$) are pairwise equal.
Then we perform the following row and column operations
\begin{eqnarray*}
R_{m_1} \leftarrow R_{m_1} - R_{m_1 -1} \\
C_{m_1} \leftarrow C_{m_1} - C_{m_1 -1}
\end{eqnarray*}
giving the rows $R_{m_1}, R_{m_1 -1}:$

\begin{center}
$ \left(\begin{array}{ccccccccccc}
  0    &  \ldots       &       0    &  d_{m_1} &    -d_{m_1}    &    0      &   \ldots         &   0   &      -1     &   \ldots  & -1\\
   0  &  \ldots      &        0  &     -d_{m_1}&    2d_{m_1}  &      0     &    \ldots      &     0   &        0    &  \ldots   &   0\\
\end{array}\right).$
\end{center}

Next we perform \begin{eqnarray*}
R_{m_1 -1} \leftarrow R_{m_1 -1}  + \frac{1}{2} R_{m_1} \\
C_{m_1 -1} \leftarrow C_{m_1 -1}  + \frac{1}{2} C_{m_1}
\end{eqnarray*}
giving

\begin{center}
$ \left(\begin{array}{ccccccccccc}
  0    &  \ldots       &       0    &  \frac{d_{m_1}}{2} &    0   &    0      &   \ldots         &   0   &      -1     &   \ldots  & -1\\
   0  &  \ldots      &        0  &      0&    2d_{m_1}  &      0     &    \ldots      &     0   &        0    &  \ldots   &   0\\
\end{array}\right).$
\end{center}

Now consider  $R_{m_1 -2}$ and $R_{m_1 -1},$ as well as its respective columns. We perform 
\begin{eqnarray*}
R_{m_1 -1} \leftarrow R_{m_1 -1} - R_{m_1 -2} \\
C_{m_1 -1} \leftarrow C_{m_1 - 1} - C_{m_1 -2}
\end{eqnarray*}
and
\begin{eqnarray*}
R_{m_1 -2} \leftarrow R_{m_1 -2}  + \frac{2}{3} R_{m_1 -1} \\
C_{m_1 -2} \leftarrow C_{m_1 -2}  + \frac{2}{3} C_{m_1 -1}
\end{eqnarray*}
obtaing $R_{m_1 -1}$ and $R_{m_1 -2}$ equal to

\begin{center}
$ \left(\begin{array}{ccccccccccc}
  0    &  \ldots       &       0    &  \frac{1}{3}    d_{m_1}&    0   &    0      &   \ldots         &   0   &      -1     &   \ldots  & -1\\
   0  &  \ldots      &        0  &      0&    \frac{3}{2}d_{m_1}  &      0     &    \ldots      &     0   &        0    &  \ldots   &   0\\
\end{array}\right).$
\end{center}

We continue this process until we annihilate the $-1$ in the first  $m_1$ rows and columns except for the first row and column. Then the first $m_1$ diagonal entries become
\begin{equation}
\label{eq2}
 \frac{d_{m_1}}{m_1},  \frac{m_1 d_{m_1}}{m_1 -1} , \frac{ (m_1 -1) d_{m_1}}{m_1 -2}, \ldots, \frac{ 3 d_{m_1}}{2}, \frac{2d_{m_1}}{1}
 \end{equation}
We repeat the same procedure to all pairwise equal rows and columns in the range of $M_{i-1} +1$ to $M_{i-1}+m_i$  for $i=2, \ldots , h,$ obtaining the matrix  with rows $R_1, R_{M_1 +1}, \ldots, \\R_{M_{h-1}+1}$
of form
\begin{center}
$   R_{M_{i-1} +1}=   \left(\begin{array}{ccccc}
\underbrace{0 \ldots 0}_{M_{i-1}} , &   \frac{d_{m_i}}{m_i} , &  \underbrace{0 \ldots 0}_{M_h  -M_{i-1} -1}  &  \underbrace{-1 \ldots -1}_{N_{h+1-i }}& \underbrace{0 \ldots 0}_{N_h   -N_{h+1-i} }    \\
\end{array}\right),$
\end{center}
for $i=1,\ldots h.$
The remaining rows if any for some $i$ contain all entries equal to zero except the diagonal ones. For each $m_i > 1$ we obtain $m_i -1$ diagonal entries equal to
(\ref{eq2}). Next by permuting the rows and columns we move rows $R_1, R_{M_1 +1}, \ldots, R_{M_{h-1}+1}$ to be in the last $h$ positions among the first $M_h$ rows. At the end of this Step 1, the partially  tridiagonal matrix 
has the following form

\begin{center}
$ \left[\begin{array}{ccccccc}
 D_1&         &     &       &      &  &\\
        &  \ddots      &     &        &  & & \\
          &          &  D_h &   &     &  & \\
          &            &     &     \frac{d_{m_1}}{m_1}   &      & B & \\
          &            &    &                                     & \ddots &     &  \\
             &            &     &  B^T                                        & &    \frac{d_{m_1}}{m_h}  & \\
                 &            &     &                                       & &     &    d_{N_h}\\
\end{array}\right],$
\end{center}
where $$ D_i = diag(\frac{m_i d_{m_i}}{m_i -1}, \ldots, \frac{2 d_{m_i}}{1}) \hspace{0,5cm} (i = 1, \ldots, h)$$
and

\begin{center}
$B= \left(\begin{array}{cccc}
 -1 \ldots -1&   -1\ldots -1      &    -1\ldots -1        &   \\
  -1 \ldots -1 &  -1 \ldots -1     &   &         \\
          &         &  &     \\
  -1 \ldots -1        &            &     &          \\
\end{array}\right),$
\end{center}
is $h\times N_h$ matrix with exactly $N_{h+1-i}$ entries $-1$ in the $i$-th row.\\

\noindent{\bf Step 2} We consider the $(h +N_h ) \times (h+N_h )$ matrix 
\begin{center}
$ \left(\begin{array}{ccccccc}
              \frac{d_{m_1}}{m_1}   &  &   &  &    &  &B \\
          &            &    \ddots   &                                     &   &     &   \\
             &            &     &                                        &     \frac{d_{m_1}}{m_h}  &  & \\
           B^T      &            &     &                                       & &     &    d_{N_h}\\
\end{array}\right).$
\end{center}
The first $n_1$ columns of $B$ are pairwise equal as well as $n_2$ and so on up to $n_h.$
We perform similar row and column operations as in Step 1 in order to remove  all $-1$  from pairwise equal columns of $B$ except the first one in the sequence.
For any $i,$ the remaining $n_i -1$ colunms if any have all entries equal to zero except to the diagonal ones equal to 
$$
  \frac{n_i d_{n_i}}{n_i -1} , \frac{ (n_i -1) d_{n_i}}{n_i -2}, \ldots, \frac{ 3 d_{n_i}}{2}, \frac{2d_{n_i}}{1}.
 $$
Again we may permute the rows and columns by pushing to the end the rows and columns with $-1$.\\

\noindent{\bf Step 3} In this stage we  consider the $2h \times 2h$ matrix 

\begin{center}
$ \left(\begin{array}{ccccccc}
              \frac{d_{m_1}}{m_1}   &  &   &  &  -1  & \ldots & -1 \\
          &            &    \ddots   &                                     &    \vdots  & &   \\
             &            &     &                                             \frac{d_{m_1}}{m_h} & -1&  & \\
            -1   &  \ldots          &     & -1                                      & \frac{d_{n_1}}{n_1}&     &    \\
                   \vdots    &      &     &                                       & &   \ddots  &    \\
                       -1   &            &     &                                       & &     & \frac{d_{n_h}}{n^{*}_h}   \\
\end{array}\right).$
\end{center}

First we note that
\begin{equation}
\label{H}
n^{*}_{h} =
\begin{cases}
         n_h -1 & \mbox{ if }  n_h > 1  \\
          n_{h-1} & \mbox{ otherwise } \\
\end{cases}
\end{equation}
since that we have eliminated the last row and column.

We  perform 
\begin{eqnarray*}
R_{h+1} \leftarrow R_{h +1}  - R_{h +2} \\
C_{h +1} \leftarrow C_{h +1}  - C_{h+2}
\end{eqnarray*}
most of $-1$ are removed from $R_{h+1}$ and $C_{h+1}.$ 
We repeat this process to the others rows and columns. At the end we obtain the $2h \times 2h$ matrix

\begin{center}
$  \left[\begin{array}{ccccccccc}
\frac{d_{m_1}}{m_1}&            &     &        &    &    &   &       -1   &     \\
&          \ddots &  &        &    &  &   &      &       \\
&            &        \frac{d_{m_h}}{m_h} &  -1 &    &    &   &     &               \\
&            &  -1   &    (\frac{d_{n_1}}{n_1} +\frac{d_{n_2}}{n_2}) &   -\frac{d_{n_2}}{n_2}     & &   &      &     \\
 &          &        &                      -\frac{d_{n_2}}{n_2}                                                 &          \ddots     & \ddots  &   &   &     \\
&           &        &                                                                                                       &                \ddots       &            &  - \frac{d_{n_{h-1}}}{n_{h -1}}  &          &\\
 &   &        &     &                                         & -\frac{d_{n_{h-1}}}{n_{h-1}}         &    (\frac{d_{n_{h-1}}}{n_{h-1}} +\frac{d_{n_h}}{n^{*}_h  })   &  - \frac{d_{n_h}}{n^{*}_h }            &\\
-1&   &      &     &        &       &     - \frac{d_{n_h}}{n^{*}_h }   &   \frac{d_{n_h}}{n^{*}_h }       &\\
\end{array}\right].$\\
\end{center}

\noindent{\bf Step 4} Since that $ \frac{d_{m_i}}{m_i} \neq 0,$ we can perform  for $i=1,\ldots,h$
\begin{eqnarray*}
R_{2h -i +1} \leftarrow R_{2h -i+1}  - \frac{ m_i}{d_{m_i}} R_{2h -i+1} \\
C_{2h -i +1} \leftarrow C_{2h -i+1}  - \frac{ m_i}{d_{m_i}} C_{2h -i+1} \\
\end{eqnarray*}
and obtaining the matrix

\begin{center}
$  \left[\begin{array}{cccccccc}
\frac{d_{m_1}}{m_1}&            &     &        &    &    &   &               \\
&          \ddots &  &        &    &  &   &           \\
&            &        \frac{d_{m_h}}{m_h} &   &    &    &   &                  \\
&            &    &    (\frac{d_{n_1}}{n_1} +\frac{d_{n_2}}{n_2} -\frac{m_h}{d_{m_h}}) &   -\frac{d_{n_2}}{n_2}     & &   &         \\
 &          &        &                      -\frac{d_{n_2}}{n_2}                                                 &          \ddots     & \ddots  &   &        \\
&           &        &                                                                                                       &                \ddots       &            &  - \frac{d_{n_{h-1}}}{n_{h-1}}  &          \\
 &   &        &     &                                         & -\frac{d_{n_{h-1}}}{n_{h-1}}         &    (\frac{d_{n_{h-1}}}{n_{h-1}} +\frac{d_{n_h}}{n^{*}_h }  -\frac{m_2}{d_{m_2}}    )     &  - \frac{d_{n_h}}{n^{*}_h }            \\
&   &      &     &        &       &     - \frac{d_{n_h}}{n^{*}_h }   &   \frac{d_{n_h}}{n^{*}_h } -\frac{m_1}{d_{m_1}}      \\
\end{array}\right].$\\
\end{center}

Our final task is to transform a tridiagonal matrix  into a upper triangular matrix.
For this, we appeal to $LU$ decomposition. Let 
\begin{center}
$T= \left[\begin{array}{ccccc}
 a_1&   b_1      &     &       &   \\
 b_1&  \ddots      &  \ddots   &        & \\
          &        \ddots  &  \ddots &   &b_{h-1}    \\
          &            &     &     b_{h-1}   &  a_{h}  \\
\end{array}\right],$
\end{center}
where  $$ a_i =  \frac{d_{n_{i}}}{n_{i}} +\frac{d_{n_{i+1}}}{n_{i+1}}  -\frac{m_{h-i+1}}{d_{m_{h-i +1}}}, b_i =  -\frac{d_{n_{i+1}}}{n_{i+1}}, \hspace{0,5cm} and \hspace{0,5cm}  a_h = \frac{d_{n_h}}{n^{*}_h} -\frac{m_1}{d_{m_1}}.$$
The $LU$ decomposition can be described as follows: for a given non singular matrix $T$ of order $h,$
we have that  $L$ and $U$ are given, for $ i=1,\ldots h-1:$

$L= \left[\begin{array}{ccccc}
 1&         &     &       &   \\
 f_1&   1      &     &       &   \\
      &    f_2     &  1   &       &   \\
     & \ddots&  \ddots  &   &     \\
          &        &       f_{h-2} & 1  &  \\
          &            &     &     f_{h-1}   &  1  \\
\end{array}\right] \hspace{0,5cm}$
$U= \left[\begin{array}{ccccc}
  g_1&  h_1       &     &       &   \\
         &   g_2      &  h_2   &       &   \\
           &         &    g_3 &  h_3     &   \\
            &        &  \ddots  &   \ddots  &     \\
          &        &                 & g_{h-1}  & h_{h-2} \\
          &            &     &                  &  g_{h}  \\
\end{array}\right],$\\
where $g_1= a_1, h_1= b_1,$ and for $i=2, \ldots h-1:$
\begin{equation}
\label{eq3}
 f_{i-1} = \frac{b_{i-1}}{g_{i-1}} \hspace{0,75cm} g_i = a_i - \frac{b_{i-1}^2 }{g_{i-1}} 
 \end{equation}

From this, we can compute the determinant of $T,$ in linear time, taking the product  of $g_i,$ since that  $det(T) = det(LU) = det(L)\cdot det(U)= \prod_{i=1}^{h} g_i.$

 We have proven the following result:

\begin{Thr}
\label{trees}
 Let  $G(m_1, \ldots, m_h ; n_1, \ldots, n_h)$ be a double nested graph of order $n,$ and $K(G)$ its Kirchhoff matrix. Then the number of spanning trees $\tau(G)$
can be  computed  in $O(n)$ time by taking the product of 
$$\frac{m_i d_{m_i}}{m_i -1}, \frac{(m_i -1) d_{m_i}}{m_i -2}, \ldots, \frac{3d_{m_i}}{2}, \frac{2 d_{m_i}}{1}, \hspace{0,5cm} \frac{n_i d_{n_i}}{n_i -1}, \frac{(n_i -1) d_{n_i}}{n_i -2}, \ldots, \frac{3d_{n_i}}{2}, \frac{2 d_{n_i}}{1},$$
$$  \frac{d_{m_1}}{m_1}. \ldots, \frac{d_{m_h}}{m_h}, \hspace{0,5cm} g_1, g_2, \ldots, g_{h}$$
where $g_1, g_2, \ldots, g_{h}$ are obtained from equation (\ref{eq3}).
\end{Thr}

\section{Bipartite with few cycles}\label{secdiag}

In this section we apply the algorithm to get the number of spanning trees  $\tau (G),$ for a some cases of  bipartite graphs $G.$
Here we consider the double nested unicyclic, bicyclic, tricyclic
graphs and the double nested quasi-tree graphs.

\subsection{Double nested unicyclic graphs}

A double nested unicyclic graph of order $n,$ is a bipartite graph  $G(1, 1 ; 2, n-4)$  with $n \geq 5.$ 
The Figure \ref{unicyclic} shows a double nested unicyclic graph.

\begin{figure}[h!]     
\begin{tikzpicture}
  [scale=1,auto=left,every node/.style={circle,scale=0.7}]
  \node[draw,circle,fill=black] (p) at (-2,4) {};
  \node[draw,circle,fill=black] (o) at (0,6) {};
  \node[draw,circle,fill=black] (n) at (-2,6) {};
   \node[draw,circle,fill=black,label=below:$n-4$] (q) at (1,4) {};
    
  \node[draw,circle,fill=black, label=left:$\ldots$] (f) at  (2,4) {};
 
  \node[draw,circle,fill=black] (c) at (0,4) {};
  \node[draw,circle,fill=black] (d) at (-1,4) {};
  \path

        (n) edge node[below]{}(p)
        (o) edge node[below]{}(p)

        (o) edge node[below]{}(d)
        (n) edge node[below]{}(d)
        
        (n) edge node[below]{}(c)
        (n) edge node[below]{}(f)
        (n) edge node[below]{}(q);

\end{tikzpicture}
       \caption{Double nested graph $G(1,1; 2, n-4)$.}
       \label{unicyclic}
\end{figure}
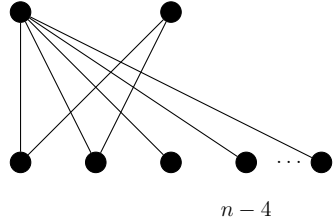

\begin{Thr}
\label{theo3}
 Let $G(1, 1 ; 2, n-4)$ be a double nested unicyclic graph of order $n \geq 5.$ Then the number of spanning trees $\tau(G)$ is equal to 4. 
\end{Thr}

\begin{proof}
 Let $G(1, 1 ; 2, n-4)$ be a double nested unicyclic graph of order $n\geq 5.$ First note that for any positive integer $p,$ we have that 
 \begin{equation}
 \label{eq4}
 \tau (G)(1,1;2,n-4+p) = \tau (G)(1,1;2, n-4),  \hspace{0,5cm} n \geq 5
 \end{equation}
 From this, taking $n=5$ and $p=1,$ 
  follows that the  vertex set $[d_{m_1}, d_{m_2} ; d_{n_1}, d_{n_2}]= [ 4, 2; 2,1]$ and $ [m_1, m_2; n_1, n^{*}_2]= [1, 1; 2, 1].$
 From Theorem \ref{trees} we have to multiply the following values 
 $$ \frac{d_{m_1}}{m_1}, \frac{d_{m_2}}{m_2}, \frac{ n_1 d_{n_1}}{n_1 -1}, g_1, g_2,$$
 where $g_2 = \frac{g_1 ( \frac{d_{n_2}}{n^{*}_2 }  -\frac{m_1}{d_{m_1}} )  -(\frac{ d_{n_2}}{n^{*}_2})^2            }{g_1} $      and $ g_1= \frac{ d_{n_1}}{n_1} + \frac{ d_{n_2}}{n^{*}_2} -\frac{m_2}{d_{m_2}}.$
 Since that $g_1 \cdot g_2 = \frac{1}{8},$
  by direct calculation from the algorithm, the number of spanning trees is given by
$$\tau (G)(1,1;2,2) =   4 \cdot 2 \cdot 4   \cdot  \frac{1}{8}   = 4.$$
Then the result follows. \end{proof}

\subsection{Double nested bicyclic graphs}

Here we consider two types of  double nested bicyclic graphs of order $n \geq 6.$ They are the following  bipartite graphs  $G(1, 1 ; 3, n-5)$  and  $H(1,2;2,n-5),$
as the Figure \ref{bicyclic} has shown.

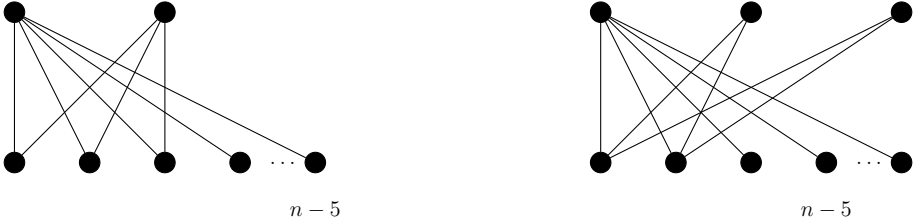
\begin{figure}[h!]
       \begin{minipage}[c]{0.3 \linewidth}
\begin{tikzpicture}
  [scale=1,auto=left,every node/.style={circle,scale=0.7}]
  \node[draw,circle,fill=black] (p) at (-2,4) {};
  \node[draw,circle,fill=black] (o) at (0,6) {};
  \node[draw,circle,fill=black] (n) at (-2,6) {};
   \node[draw,circle,fill=black] (q) at (1,4) {};
    
  \node[draw,circle,fill=black, label=left:$\ldots$, label=below:$n-5$] (f) at  (2,4) {};
 
  \node[draw,circle,fill=black] (c) at (0,4) {};
  \node[draw,circle,fill=black] (d) at (-1,4) {};
  \path

        (n) edge node[below]{}(p)
        (o) edge node[below]{}(p)

        (o) edge node[below]{}(d)
         (o) edge node[below]{}(c)
        (n) edge node[below]{}(d)
        
        (n) edge node[below]{}(c)
        (n) edge node[below]{}(f)
        (n) edge node[below]{}(q);

\end{tikzpicture}
       \end{minipage}\hfill
       \begin{minipage}[l]{0.5 \linewidth}
\begin{tikzpicture}
  [scale=1,auto=left,every node/.style={circle,scale=0.7}]
  \node[draw,circle,fill=black] (p) at (-2,4) {};
  \node[draw,circle,fill=black] (o) at (0,6) {};
  \node[draw,circle,fill=black] (s) at (2,6) {};
  \node[draw,circle,fill=black] (n) at (-2,6) {};
   \node[draw,circle,fill=black,label=below:$n-5$] (q) at (1,4) {};
    
  \node[draw,circle,fill=black, label=left:$\ldots$] (f) at  (2,4) {};
 
  \node[draw,circle,fill=black] (c) at (0,4) {};
  \node[draw,circle,fill=black] (d) at (-1,4) {};
  \path

        (n) edge node[below]{}(p)
        (o) edge node[below]{}(p)
        
        (s) edge node[below]{}(p)
       (s) edge node[below]{}(d)

        (o) edge node[below]{}(d)
        (n) edge node[below]{}(d)
        
        (n) edge node[below]{}(c)
        (n) edge node[below]{}(f)
        (n) edge node[below]{}(q);

\end{tikzpicture}
       \end{minipage}
       \caption{  $G(1,1;3, n-5)$  and $H(1,2;2,n-5).$  }
              \label{bicyclic}
\end{figure}

\begin{Thr}
\label{theo4}
 Let be  $G(1, 1 ; 3, n-5)$  and  $H(1,2;2,n-5),$ the double nested bicyclic graph of order $n \geq 6.$ Then the number of spanning trees $\tau(G)$  and $\tau (H)$ is equal to 12. 
\end{Thr}

\begin{proof} If $G(1,1;3,n-5)$ the proof is similar to Theorem \ref{theo3}. Now 
 let $H(1, 2 ; 2, n-5)$ be a double nested bicyclic graph of order $n\geq 6.$  Since that for any positive integer $p$ 
 \begin{equation}
 \label{eq5}
 \tau (H)(1,2;2,n-5+p) = \tau (H)(1,2;2, n-5),  \hspace{0,5cm} n \geq 6
 \end{equation}
 From this, taking $n=6$ and $p=1,$ 
  follows that the  vertex set $[d_{m_1}, d_{m_2} ; d_{n_1}, d_{n_2}]= [ 4, 2; 3,1]$ and $ [m_1, m_2; n_1, n^{*}_2]= [1, 2; 2, 1].$
 From Theorem \ref{trees} we have to multiply the following values 
 $$ \frac{d_{m_1}}{m_1}, \frac{d_{m_2}}{m_2}, \frac{ m_2 d_{m_2}}{m_2 -1}, \frac{ n_1 d_{n_1}}{n_1 -1}, g_1, g_2,$$
 where $g_2 = \frac{g_1 ( \frac{d_{n_2}}{n^{*}_2 }  -\frac{m_1}{d_{m_1}} )  -(\frac{ d_{n_2}}{n^{*}_2})^2            }{g_1} $      and $ g_1= \frac{ d_{n_1}}{n_1} + \frac{ d_{n_2}}{n^{*}_2} -\frac{m_2}{d_{m_2}}.$
 Since that $g_1 \cdot g_2 = \frac{1}{8},$
  by direct calculation from the algorithm, the number of spanning trees is given by
$$\tau (H)(1,2;2,2) =   4 \cdot 1 \cdot 4   \cdot 6 \cdot \frac{1}{8}   = 12.$$
Then the result follows. \end{proof}

\subsection{Double nested tricyclic graphs}

Here we consider three types of  double nested tricyclic graphs of order $n \geq 7.$ They are the following  bipartite graphs  $G(1, 1 ; 4, n-6),  H(1,1,1;2,1,n-6)$   and  $S(1,3;2,n-6),$
as the Figure \ref{tricyclic} has shown.

\begin{figure}[h!]
       \begin{minipage}[c]{0.2 \linewidth}
\begin{tikzpicture}
  [scale=1,auto=left,every node/.style={circle,scale=0.7}]
  \node[draw,circle,fill=black] (p) at (-2,4) {};
  \node[draw,circle,fill=black] (o) at (0,6) {};
  \node[draw,circle,fill=black] (n) at (-2,6) {};
   \node[draw,circle,fill=black] (q) at (1,4) {};
    
  \node[draw,circle,fill=black, label=right:$\ldots$] (f) at  (2,4) {};
  
    \node[draw,circle,fill=black,  label=below:$n-6$] (g) at  (3,4) {};

  \node[draw,circle,fill=black] (c) at (0,4) {};
  \node[draw,circle,fill=black] (d) at (-1,4) {};
  \path

        (n) edge node[below]{}(p)
        (o) edge node[below]{}(p)
       
         (o) edge node[below]{}(q)
         (n) edge node[below]{}(g)

        (o) edge node[below]{}(d)
         (o) edge node[below]{}(c)
        (n) edge node[below]{}(d)
        
        (n) edge node[below]{}(c)
        (n) edge node[below]{}(f)
        (n) edge node[below]{}(q);

\end{tikzpicture}
       \end{minipage}\hfill
       \begin{minipage}[l]{0.5 \linewidth}
\begin{tikzpicture}
  [scale=1,auto=left,every node/.style={circle,scale=0.7}]
  \node[draw,circle,fill=black] (p) at (-2,4) {};
  \node[draw,circle,fill=black] (o) at (0,6) {};
  \node[draw,circle,fill=black] (s) at (2,6) {};
  \node[draw,circle,fill=black] (n) at (-2,6) {};
   \node[draw,circle,fill=black] (q) at (1,4) {};
    
  \node[draw,circle,fill=black, label=left:$\ldots$, label=below:$n-6$] (f) at  (2,4) {};
 
  \node[draw,circle,fill=black] (c) at (0,4) {};
  \node[draw,circle,fill=black] (d) at (-1,4) {};
  \path

        (n) edge node[below]{}(p)
        (o) edge node[below]{}(p)
        
        (s) edge node[below]{}(p)
       (s) edge node[below]{}(d)
       
        (s) edge node[below]{}(c)

        (o) edge node[below]{}(d)
        (n) edge node[below]{}(d)
        
        (n) edge node[below]{}(c)
        (n) edge node[below]{}(f)
        (n) edge node[below]{}(q);

\end{tikzpicture}
       \end{minipage}
       \begin{minipage}[l]{0.3 \linewidth}
\begin{tikzpicture}
  [scale=1,auto=left,every node/.style={circle,scale=0.7}]
  \node[draw,circle,fill=black] (p) at (-2,4) {};
  \node[draw,circle,fill=black] (o) at (0,6) {};
  \node[draw,circle,fill=black] (t) at (4,6) {};

  \node[draw,circle,fill=black] (s) at (2,6) {};
  \node[draw,circle,fill=black] (n) at (-2,6) {};
   \node[draw,circle,fill=black,label=below:$n-6$] (q) at (1,4) {};
    
  \node[draw,circle,fill=black, label=left:$\ldots$] (f) at  (2,4) {};
 
  \node[draw,circle,fill=black] (c) at (0,4) {};
  \node[draw,circle,fill=black] (d) at (-1,4) {};
  \path

        (n) edge node[below]{}(p)
        (o) edge node[below]{}(p)
        
        (s) edge node[below]{}(p)
       (s) edge node[below]{}(d)
       
        (t) edge node[below]{}(p)
      
       (t) edge node[below]{}(d)

        (o) edge node[below]{}(d)
        (n) edge node[below]{}(d)
        
         (n) edge node[below]{}(c)
       
        (n) edge node[below]{}(f)
        (n) edge node[below]{}(q);

\end{tikzpicture}
       \end{minipage}

       \caption{ The bipartite tricyclic graphs.  }
              \label{tricyclic}
\end{figure}
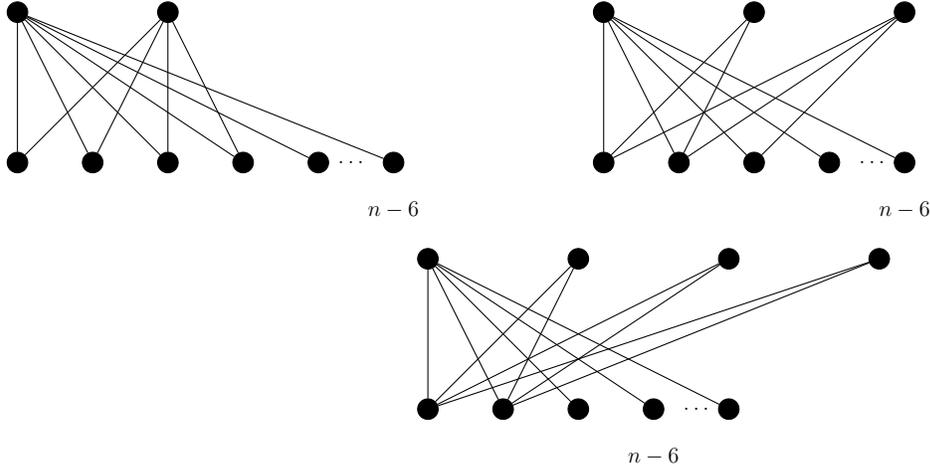

\begin{Thr}
 Let be  $G(1, 1 ; 4, n-6), H(1, 1, 1; 2,1, n-6) $  and  $S(1,3;2,n-6),$ the double nested tricyclic graph of order $n \geq 7.$ Then the number of spanning trees $\tau(G), \tau(H)$  and $\tau (S)$ is equal to 32, 36 and 32, respectively. 
\end{Thr}
\begin{proof}
 The proof for $G(1, 1 ; 4, n-6)$ and $S(1,3;2,n-6)$  is similar to Theorem \ref{theo3} and Theorem \ref{theo4}. Now let be $H(1, 1 ,1; 2,1, n-6)$ of order $n\geq 7.$ 
 Since that for any positive integer $p$ 
 \begin{equation}
 \label{eq5}
 \tau (H)(1,1,1;2,1, n-6+p) = \tau (H)(1,1, 1;2, 1, n-6),  \hspace{0,5cm} n \geq 7
 \end{equation}
 From this, taking $n=7$ and $p=1,$ 
  follows that the  vertex set $[d_{m_1}, d_{m_2} , d_{m_3}; d_{n_1}, d_{n_2}, d_{n_3}]= [ 5,3, 2; 3,2,1]$ and $ [m_1, m_2, m_3, n_1, n_2, n^{*}_3]= [1, 1,1; 2, 1,1].$
 From Theorem \ref{trees} we have to multiply the following values 
 $$ \frac{d_{m_1}}{m_1}, \frac{d_{m_2}}{m_2}, \frac{d_{m_3}}{m_3}, \frac{ n_1 d_{n_1}}{n_1 -1}, g_1, g_2, g_3,$$
 where $g_1= 3, g_2 = \frac{4}{3},$ and $g_3 = \frac{1}{20}.$
 Since that $g_1 \cdot g_2 \cdot g_3 = \frac{1}{5},$
  by direct calculation from the algorithm, the number of spanning trees is given by
$$\tau (H)(1,1,1;2,1,2) =   5 \cdot 3 \cdot 2  \cdot 6 \cdot \frac{1}{5}   = 36.$$
Then the result follows. 
\end{proof}

\subsection{Double nested quasi-tree graphs}

A connected graph $G$ is called a quasi-tree graph if there  exist  $v_0 \in V(G)$ such that $G - v_0$ is a tree.
Let be the  quasi-tree graph $G(1,1 ;d_0, n -d_0 -2)$  with $ 2 \leq d_0 <  n-2.$
The Figure \ref{quasi} shows a  quasi-tree graph.

\begin{figure}[h!]     
\begin{tikzpicture}
  [scale=1,auto=left,every node/.style={circle,scale=0.7}]
  \node[draw,circle,fill=black] (p) at (-2,4) {};
  \node[draw,circle,fill=black, label=right:$v_0$] (o) at (0,6) {};
  \node[draw,circle,fill=black] (n) at (-2,6) {};
   \node[draw,circle,fill=black] (q) at (1,4) {};
    
  \node[draw,circle,fill=black, label=left:$\ldots$, label=below:$n-d_0 -2$] (f) at  (2,4) {};
 
  \node[draw,circle,fill=black] (c) at (0,4) {};
  \node[draw,circle,fill=black, label=below:$d_0$] (d) at (-1,4) {};
  \path

        (n) edge node[below]{}(p)
        (o) edge node[below]{}(p)

        (o) edge node[below]{}(d)
         (o) edge node[below]{}(c)
        (n) edge node[below]{}(d)
        
        (n) edge node[below]{}(c)
        (n) edge node[below]{}(f)
        (n) edge node[below]{}(q);

\end{tikzpicture}
       \caption{Quasi-tree graph $G(1,1; d_0, n-d_0-2)$.}
       \label{quasi}
\end{figure}
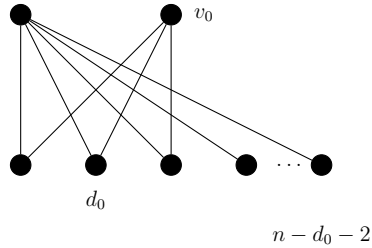

\begin{Thr}
 Let   $G(1,1; d_0, n -d_0 -2)$ be the quasi-tree graph  with $2\leq d_0 < n .$ Then the number of spanning trees $\tau(G)$    is equal to $  2^{d_0 -1}  d_0 .$   
\end{Thr}
Since that the proof is  similar to the results above, we omit them.

\end{document}